\documentclass[12pt]{amsart}

\def\@typesizes{%
       \or{5}{6.5}\or{6}{7.5}\or{7}{8.5}\or{8}{11}\or{9}{12}%
       \or{10}{13}% normalsize
       \or{\@xipt}{14}\or{\@xiipt}{15}\or{\@xivpt}{18}%
       \or{\@xviipt}{20}\or{\@xxpt}{24}}

\usepackage[utf8]{inputenc}
\usepackage[T1]{fontenc}

\usepackage{hyperref}
\usepackage{graphicx}
\usepackage[english,ukrainian]{babel}
\usepackage{amsmath,amsfonts,amssymb,amsthm,cite}
\numberwithin{equation}{section}
\newtheorem{theorem}{Theorem}[section]
\newtheorem{corollary}[theorem]{Corollary}

\newtheorem{remark}[theorem]{Remark}

\newtheorem{definition}[theorem]{Definition}
\newcommand{\nonprint}[1]{}

\begin{document}

\title[Continuity in a parameter]{Continuity in a parameter of solutions to boundary-value problems in Sobolev spaces}

\author{Olena Atlasiuk}
\address{
University of Helsinki, Department of Mathematics and Statistics, P.O. Box 68, Pietari Kalmin katu 5, 00014 Helsinki, Finland and \\
Institute of Mathematics of the National Academy of Sciences of Ukraine, st. Tereschenkivska 3, 01024 Kyiv, Ukraine }

\email{olena.atlasiuk@helsinki.fi}

\author{Vladimir Mikhailets}
 \address{ing's College London, Strand, WC2R 2LS London, UK and Institute of Mathematics of the National Academy of Sciences of Ukraine, st. Tereschenkivska 3, 01024 Kyiv, Ukraine (mikhailets@imath.kiev.ua}

\email{mikhailets@imath.kiev.ua}

\begin{abstract}
We study the most general class of linear inhomogeneous boundary-value problems for systems of ordinary differential equations of an arbitrary order whose solutions and right-hand sides belong to appropriate Sobolev spaces. For parameter-dependent problems from this class, we prove a constructive criterion for their solutions to be continuous in the Sobolev space with respect to the parameter. We also prove a two-sided estimate for the degree of convergence of these solutions to the solution of the nonperturbed problem.
\end{abstract}

\maketitle

\textbf{Keywords:} differential system, boundary-value problem, Sobolev space, continuity in parameter.

2020 Mathematics Subject Classification: 34B05, 34B08, 47A53

\textit{The authors thank Prof. Aleksandr Murach for his discussion of the paper and valuable remarks.}

\maketitle

\tableofcontents

\section{Introduction}\label{Sec.1}
The investigation of the solutions of systems of ordinary differential equations is an important part of numerous problems of contemporary analysis and its applications (see, e.g., \cite{BochSAM2004} and the references therein). Questions concerning limit transition in parameter-dependent differential equations arise in various mathematical problems. I.I. Gikhman \cite{Gikhman}, M.A. Krasnoselskii and S.G. Krein \cite{KrasnKrein}, J. Kurzweil and Z. Vorel \cite{KurzweilVorel} obtained fundamental results on the continuity with respect to the parameter of solutions to the Cauchy problem for nonlinear differential systems. For linear systems, these results were refined and supplemented by A.Yu. Levin \cite{Levin}, Z. Opial \cite{Opial}, W.T. Reid \cite{Reid}, and T.K. Nguen \cite{Nguen}.

Unlike Cauchy problems, the solutions to such problems may not exist or may not be unique. Thus, it is interesting to investigate the character of the solvability of inhomogeneous boundary-value problems in Sobolev spaces and the dependence of their solutions on the parameter \cite{GnypKodliukMikhailets2015, HnypMikhailetsMurach2017, KodliukMikhailets2013}. These questions are studied in the best way for ordinary differential systems of the first order. I.T. Kiguradze \cite{Kigyradze2003} and then M. Ashordia \cite{Ashordia} introduced and investigated a class of general linear boundary-value problems for systems of first order differential equations. For general linear boundary-value problems, the conditions required for the Fredholm property and the continuous dependence of the solutions on parameters were established by I.T. Kiguradze \cite{Kigyradze1987, Kigyradze1975}.

V.A. Mikhailets and his followers introduced and studied the generic classes of boundary-value problems for systems of ordinary differential equations with respect to the Sobolev spaces or to the spaces of continuously differentiable functions. They proved that such problems are Fredholm, and obtained criterion for their well-posedness and continuity in the parameter of their solutions in these spaces. These results have been used for the investigation of multipoint boundary-value problems \cite{Atl3, Atl4, AtlasiukMikhailets2024}, Green's matrices, and used in the spectral theory of differential operators with singular coefficients \cite{GoriunovMikhailetsPankrashkin2013EJDE, GoriunovMikhailets2010MFAT, GoriunovMikhailets2012UMJ, GoriunovMikhailets2010MN}.

The most general classes of linear boundary-value problems are considered in \cite{AtlasiukMikhailets20191, AtlasiukMikhailets20192, AtlasiukMikhailets2024, AtlasiukMikhailets20194, MikhAtlSkor}. These classes relate to the classical scale of complex Sobolev spaces and are introduced for the systems whose right-hand sides and solutions run through the corresponding Sobolev spaces. The boundary conditions for these systems are given in the most general form by means of an arbitrary continuous linear operator given on the Sobolev space of the solutions. Therefore, it is natural to say that these boundary-value problems are generic with respect to the corresponding Sobolev space. Since the usual methods of the theory of ordinary differential equations are not applicable to these problems, their study is of a special interest and requires new approaches and methods.
The present paper extends these results onto differential systems of an arbitrary order. In contrast to the method of \cite{HnypMikhailetsMurach2017}, our approach is general and allows us to investigate solutions to boundary-value problems not only in Sobolev spaces of integer order \cite{MikhAtlSkor, MikhMurachSol2016}.

The paper is organised as follows. In Section~\ref{Sec.1}, we recall known facts concerning continuity in a parameter of solutions to boundary-value problems. In Section~\ref{Sec.2}, we give statement of the problem, that is, the inhomogeneous boundary-value problem \eqref{bound_pr_1}, \eqref{bound_pr_2}. %using the object of our study
In Section~\ref{Sec.3}, we formulate a constructive criterion under which the solutions of parameter-dependent problems are continuous in the Sobolev space with respect to the parameter. Besides, we establish a two-sided estimate for the degree of convergence of the solutions.
In Section~\ref{Sec.4}, we prove our results.

\section{Statement of the problem}\label{Sec.2}

Let $(a,b)\subset\mathbb{R}$ be a compact interval and suppose parameters $$\{m, n+1, r\} \subset \mathbb{N} \quad \mbox{and} \quad 1\leq p\leq\infty.$$
We use the complex Sobolev space $W_p^n:=W_p^n([a,b];\mathbb{C})$ and set $W_p^{0}:=L_p$. By $$(W_p^n)^{m}:=W_p^n([a,b];\mathbb{C}^{m}) \quad\mbox{and} \quad (W_p^n)^{m\times m}:=W_p^n([a,b];\mathbb{C}^{m\times m})$$ we denote the Sobolev spaces of vector-valued and matrix-valued functions, respectively, whose entries belong to the Sobolev space $W_p^n$ of scalar functions on $(a,b)$, with the vectors having $m$ entries and with the matrices being of $m\times m$ type. By $\|\cdot\|_{n, p}$ we denote the norms in these spaces. They are the sums of the corresponding norms in $W_p^n$ of all components of vector-valued or matrix-valued functions from these spaces. It will be always clear from context in which Sobolev space (scalar or vector-valued or matrix-valued functions) these norms are considered. For $m=1$, all these spaces coincide. It is known that the spaces   are Banach spaces. If $p<\infty$, they are separable and have Schauder bases.

We consider the following inhomogeneous boundary-value problem for a system of $m$ linear differential equations of order $r$:
\begin{equation}\label{bound_pr_1}
(Ly)(t):=y^{(r)}(t) + \sum\limits_{j=1}^rA_{r-j}(t)y^{(r-j)}(t)=f(t), \quad t\in(a,b),
\end{equation}
\begin{equation}\label{bound_pr_2}
By=c.
\end{equation}
Here, the matrix-valued functions $A_{r-j}(\cdot)$ belong to the space $(W_p^n)^{m\times m}$, vector-valued function $f(\cdot)$ belongs to the space $(W^n_p)^m$, vector $c$ belongs to the space $\mathbb{C}^{rm}$, and $B\colon (W^{n+r}_p)^m\rightarrow\mathbb{C}^{rm}$ is a linear continuous operator.

A solution of the boundary-value problem \eqref{bound_pr_1}, \eqref{bound_pr_2} is understood as a vector-valued function $y(\cdot)\in (W_{p}^{n+r})^m$ that satisfies equation \eqref{bound_pr_1} (everywhere for $n\geq 1$, and almost everywhere for $n=0$) on $(a,b)$ and equality \eqref{bound_pr_2} (which means $rm$ scalar boundary conditions). Indeed, if the right-hand side $f(\cdot)$ of the system runs through the whole space $(W_{p}^{n})^m$, then the solution $y(\cdot)$ to the system runs through the whole space $(W^{n+r}_p)^m$. The boundary condition \eqref{bound_pr_2} is the most general for the differential system \eqref{bound_pr_1}.

Let the parameter $\mu$ runs through the close interval $I \subset \mathbb{R}$. We consider the following parameter-dependent inhomogeneous boundary-value problem of the form \eqref{bound_pr_1}, \eqref{bound_pr_2} for a system of $m$ linear differential equations of order $r$:
\begin{equation}\label{bound_z1}
L(\mu)y(t,\mu):=y^{(r)}(t,\mu) + \sum\limits_{j=1}^rA_{r-j}(t,\mu)y^{(r-j)}(t,\mu)=f(t,\mu), \qquad t\in(a,b),
\end{equation}
\begin{equation}\label{bound_z2}
B(\mu)y(\cdot,\mu)=c(\mu).
\end{equation}
Here, for every $\mu \in I$, the unknown vector-valued function $y(\cdot,\mu)$ belongs to the space $(W^{n+r}_p)^m$, and we arbitrarily choose the matrix-valued functions $A_{r-j}(\cdot,\mu)\in (W_p^n)^{m\times m}$, with $j\in \{1, \ldots, r\}$, vector-valued function $f(\cdot,\mu)\in (W^n_p)^m$, vector $c(\mu)\in \mathbb{C}^{rm}$, and linear continuous operator
\begin{equation}\label{oper_B(e)v}
B(\mu)\colon (W^{n+r}_p)^m\rightarrow\mathbb{C}^{rm}.
\end{equation}

We interpret vectors and vector-valued functions as columns. Note that the functions $A_{r-j}(t,\mu)$ are not assumed to have any regularity with respect to $\mu$.

Let us specify the sense in which equation \eqref{bound_z1} is understood. A solution of the boundary-value problem \eqref{bound_z1}, \eqref{bound_z2} is understood as a vector-valued function $y(\cdot,\mu)\in (W^{n+r}_p)^m$ that satisfies equation \eqref{bound_z1} (everywhere for $n\geq 1$, and almost everywhere for $n=0$) on $(a,b)$ and equality \eqref{bound_z2} (which means $rm$ scalar boundary conditions). The boundary condition \eqref{bound_z2} with an arbitrary continuous operator \eqref{oper_B(e)v} is the most general for the differential system \eqref{bound_z1}. Indeed, if the right-hand side $f(\cdot,\mu)$ of the system runs through the whole space $(W_{p}^{n})^m$, then the solution $y(\cdot,\mu)$ to the system runs through the whole space $(W^{n+r}_p)^m$. This condition covers all the classical types of boundary conditions such as initial conditions in the Cauchy problem, various multipoint conditions, integral conditions, conditions used in mixed boundary-value problems, and also nonclassical conditions containing the derivatives, generally fractional, and the order of these derivatives may exceed the order of the differential equation. Therefore, the boundary-value problem \eqref{bound_z1}, \eqref{bound_z2} is generic with respect to the Sobolev space $W^{n+r}_p$.

\section{Main results}\label{Sec.3}

With the boundary-value problem \eqref{bound_z1}, \eqref{bound_z2}, we associate the linear continuous operator
\begin{equation}\label{(L,B)vp}
\left(L(\mu),B(\mu)\right) \colon (W^{n+r}_p)^m\to (W^{n}_p)^m\times\mathbb{C}^{rm}.
\end{equation}

Recall that a linear continuous operator $T\colon X \rightarrow Y$, where $X$ and $Y$ are Banach spaces, is called a Fredholm operator if its kernel $\ker T$ and cokernel $Y/T(X)$ are finite-dimensional. If operator $T$ is a Fredholm one, then its range $T(X)$ is closed in $Y$ (see, e. g., \cite[Lemma~19.1.1]{Hermander1985}). The finite index of the Fredholm operator $T$ is defined by the formula
$$
\mathrm{ind}\,T:=\dim\ker T-\dim(Y/T(X)) \in \mathbb{Z}.
$$

According to \cite[Theorem 2.1]{AtlasiukMikhailets20194}, operator \eqref{(L,B)vp} is a Fredholm one with zero index for every $\mu \in I$.

Let us consider the following condition for a fixed point $\mu_0 \in I$.

\textbf{Condition (0)}. The homogeneous boundary-value problem
\begin{equation*}
L(\mu_0)y(t;\mu_0)=0,\quad t\in (a,b),\quad B(\mu_0)y(\cdot;\mu_0)=0
\end{equation*}
has only the trivial solution.

Let us now give our basic concepts.

\begin{definition}\label{defin_vp} The solution to the boundary-value problem \eqref{bound_z1}, \eqref{bound_z2} depends continuously on the parameter $\mu$ at $\mu_0 \in I$ if the following two conditions are satisfied:
\begin{itemize}
\item[$(\ast)$] There exists a positive number $\varepsilon$ such that, for any $\mu\in (\mu_0-\varepsilon, \mu_0+\varepsilon)$ and arbitrary right-hand sides $f(\cdot;\mu)\in (W^{n}_p)^{m}$ and $c(\mu)\in\mathbb{C}^{rm}$, this problem has a unique solution $y(\cdot;\mu)$ from the space  $(W^{n+r}_p)^{m}$;
\item [$(\ast\ast)$] The convergence of the right-hand sides $f(\cdot;\mu)\to f(\cdot;\mu_0)$ in $(W_p^n)^{m}$ and $c(\mu)\to c(\mu_0)$ in $\mathbb{C}^{rm}$ as $\mu\to\mu_0$ implies the convergence of the solutions
\begin{equation*}\label{4.guv}
y(\cdot;\mu)\to y(\cdot;\mu_0)\quad\mbox{in}\quad (W^{n+r}_p)^{m}
\quad\mbox{as}\quad\mu\to\mu_0.
\end{equation*}
\end{itemize}
\end{definition}

We also consider the next two conditions on the left-hand sides of this problem.

\textbf{Limit Conditions} as $\mu\to\mu_0$:
\begin{itemize}
  \item [(I)] $A_{r-j}(\cdot;\mu)\to A_{r-j}(\cdot;\mu_0)$ in the space $(W^{n}_p)^{m\times m}$ for each number  $j\in\{1,\ldots, r\}$;
  \item [(II)] $B(\mu)y\to B(\mu_0)y$ in the space $\mathbb{C}^{m}$ for every $y\in(W^{n+r}_p)^m$.
\end{itemize}

Let us formulate the main result of the paper.

\begin{theorem}\label{nep v}
The solution to the boundary-value problem \eqref{bound_z1}, \eqref{bound_z2} depends continuously on the parameter~$\mu$ at $\mu_0 \in I$ if and only if this problem satisfies Condition \textup{(0)} and Limit Conditions~\textup{(I)} and~\textup{(II)}.
\end{theorem}

\begin{corollary}\label{nep v}
If Condition \textup{(0)} and Limit Conditions~\textup{(I)} and~\textup{(II)} are satisfied for each $\mu_0 \in I$, then the solution to the boundary-value problem \eqref{bound_z1}, \eqref{bound_z2} exists and is unique for arbitrary right-hand sides of the problem and belongs to the space $C\big(I; (W^{n+r}_p)^{m}\big)$.
\end{corollary}

\begin{remark}
In the case of $r=1$, $I=[0, \varepsilon_0]$, $\mu_0=0$, Theorem \ref{nep v} is proved in \cite[Theorem 1]{AtlasiukMikhailets20192}.

Paper \cite{HnypMikhailetsMurach2017} gives us a constructive criterion under which the solutions to parameter-dependent problems are continuous with respect to the small parameter in the Sobolev spaces $W_{p}^{n}$, where $1\leq p<\infty$. The proof of the criterion is based on the fact that the continuous linear operator $B$, for every $\mu \in I$ and $1\leq p<\infty$, admits the following unique analytic representation
\begin{gather*}\label{st anal vp}
By=\sum _{s=0}^{n+r-1} \alpha_{s}y^{(s)}(a)+\int_{a}^b \Phi(t)y^{(n+r)}(t){\rm d}t, \quad y(\cdot)\in (W_{p}^{n+r})^{m}.
\end{gather*}
Here, the matrices $\alpha_{s}$ belong to the space $\mathbb{C}^{rm\times m}$, and the matrix-valued function $\Phi(\cdot)$ belongs to the space $L_{p^{'}}\big([a, b]; \mathbb{C}^{rm\times m}\big)$, with  $1/p + 1/p^{'}=1$. For $p=\infty$, this formula also defines a continuous operator $B \colon (W_{\infty}^{n+r})^{m} \rightarrow \mathbb{C}^{rm}$. However, there exist continuous operators from $(W_{\infty}^{n+r})^{m}$ to $\mathbb{C}^{rm}$ specified by the integrals over finitely additive measures \cite{Dunford}.

Our method allows to investigate such problems in the Sobolev spaces $W_{p}^{n}$, where $1\leq p\leq \infty$, and some other function spaces (see \cite{Hnyp, MikhMurachSol2016}).
\end{remark}

We supplement our result with a two-sided estimate of the error $$\bigl\|y(\cdot;\mu_0)-y(\cdot;\mu)\bigr\|_{n+r,p}$$ of the solution $y(\cdot;\mu)$ via its discrepancy
\begin{equation*}\label{nevyuzka v}
\widetilde{d}_{n,p}(\mu):=
\bigl\|L(\mu)y(\cdot;\mu_0)-f(\cdot;\mu)\bigr\|_{n,p}+
\bigl\|B(\mu)y(\cdot;\mu_0)-c(\mu)\bigr\|_{\mathbb{C}^{rm}}.
\end{equation*}
Here, we interpret $y(\cdot;\mu)$ as an approximate solution to the problem \eqref{bound_z1}, \eqref{bound_z2}.

\begin{theorem}\label{3.6.th-bound v}
Let the boundary-value problem \eqref{bound_z1}, \eqref{bound_z2} satisfy Conditions \textup{(0)} and Limit Conditions \textup{(I)} and \textup{(II)}. Then there exist positive number $\varepsilon$, $\gamma_{1}$, and $\gamma_{2}$ such that
\begin{equation}\label{3.6.bound}
\begin{aligned}
\gamma_{1}\,\widetilde{d}_{n,p}(\mu)
\leq\bigl\|y(\cdot;\mu_0)-y(\cdot;\mu)\bigr\|_{n+r,p}\leq
\gamma_{2}\,\widetilde{d}_{n,p}(\mu),
\end{aligned}
\end{equation}
for any $\mu\in(\mu_0-\varepsilon,\mu_0+\varepsilon)$. Here, the numbers $\varepsilon$, $\gamma_{1}$, and $\gamma_{2}$ do not depend on $y(\cdot;\mu_0)$, and $y(\cdot;\mu)$.
\end{theorem}

Thus, the error and discrepancy of the solution to the problem \eqref{bound_z1}, \eqref{bound_z2} have the same degree of smallness.

\section{Proofs of main results}\label{Sec.4}

\begin{proof}[Proof of Theorem \ref{nep v}]

Let us prove \emph{the necessity}. Suppose that the boundary-value problem \eqref{bound_z1}, \eqref{bound_z2} satisfies Definition \ref{defin_vp}. Then the problem obviously satisfies Condition~(0). It remains to prove that the problem satisfies Limit Conditions (I) and~(II). We divide our proof into three steps.

\emph{Step 1.} Let us prove that the problem \eqref{bound_z1}, \eqref{bound_z2} satisfies Limit Condition \textup{(I)}. For this purpose, we canonically reduce this system to a certain system of the first order differential equations (see, e.g., \cite{Cartan1971}). We put
\begin{gather}\label{yx}
x(\cdot,\mu):=\mathrm{col}\bigl(y(\cdot,\mu),y'(\cdot, \mu),\ldots,y^{(r-1)}(\cdot, \mu)\bigr)\in(W^{n+r}_p)^{rm},\\
g(\cdot,\mu):=\mathrm{col}\bigl(\underbrace{0, \dots, 0}_{(r-1)m},f(\cdot, \mu)\bigr)
\in(W^{n}_p)^{rm},\notag\\
c(\mu):=\mathrm{col}\bigl(c_{1}(\mu),\ldots,c_{r}(\mu)\bigr)
\in\mathbb{C}^{rm}\notag,
\end{gather}
where $\mathrm{col}\bigl(\cdot,\ldots, \cdot)$ means a column vector, and define the block matrix-valued function $K(\cdot,\mu) \in (W^{n}_p)^{rm\times rm}$ as equality
\begin{equation*}\label{AAv-reduced}
K(\cdot,\mu):=\left(
\begin{array}{ccccc}
O_m & -I_m & O_m & \ldots & O_m \\
O_m & O_m & -I_m & \ldots & O_m \\
\vdots & \vdots & \vdots & \ddots & \vdots \\
O_m & O_m & O_m & \ldots & -I_m \\
A_0(\cdot,\mu) & A_1(\cdot,\mu) & A_2(\cdot,\mu) & \ldots & A_{r-1}(\cdot,\mu)\\
\end{array}\right).
\end{equation*}

Vector-valued function $y(\cdot,\mu)\in(W^{n+r}_p)^{m}$ is a solution to the system \eqref{bound_z1} if and only if vector-valued function \eqref{yx} is a solution to the system
\begin{gather*}\label{Cauchi-reduced-DE}
x'(t,\mu)+K(t,\mu)x(t,\mu)=g(t,\mu),
\quad t\in(a,b).
\end{gather*}

We denote by
\begin{equation*}\label{matrix_BYe}
\left[B(\mu)Y(\cdot,\mu)\right]:=\left(\left[B(\mu)Y_0(\cdot,\mu)\right],\dots,\left[B(\mu)
Y_{r-1}(\cdot,\mu)\right]\right) \in \mathbb{C}^{rm\times m}
\end{equation*}
the block numerical matrix of $rm\times m$-dimension. It consists of $r$ square block columns $\left[B(\mu)Y_l(\cdot,\mu)\right]\in \mathbb{C}^{m\times m}$ in which $j$-th column of the matrix $\left[B(\mu)Y_l(\cdot,\mu)\right]$ is the result of the action of the operator $B(\mu)$ in \eqref{oper_B(e)v} on the $j$-th column of the matrix-valued function $Y_l(\cdot,\mu)$.

Consider the following matrix boundary-value problem:
\begin{gather}%\label{eq1vk}
Y_l^{(r)}(t,\mu)+\sum\limits_{j=1}^{r}{A_{r-j}(t,\mu)Y_l^{(r-j)}(t,\mu)}=O_{m\times rm},
\quad t\in (a,b),\nonumber \\
\left[B(\mu)Y(\cdot,\mu)\right]=I_{rm}\label{eq2vk}.
\end{gather}
Here,
$$
Y_l(\cdot,\mu):=\left(y_l^{j,k}(\cdot,\mu)\right)_{\substack{j=1,\ldots,m\\ k=1,\ldots,rm}}
$$
is an unknown $m \times rm$ matrix-valued function with entries from $(W^{n+r}_p)^{m \times rm}$, $O_{m \times rm}$ is the zero matrix, and $I_{rm}$ is the identity matrix. This problem is a collection of $rm$ boundary-value problems \eqref{bound_z1}, \eqref{bound_z2} the right-hand sides of which do not depend on $\mu$.

Therefore, this problem has a unique solution $Y(\cdot;\mu)$  for every $\mu\in (\mu_0-\varepsilon,\mu_0+\varepsilon)$ due to condition $(\ast)$ of Definition \ref{defin_vp}. Moreover, due to condition $(\ast\ast)$ of this definition, we have the following convergence
\begin{equation}\label{zb ym}
y_{j,k}(\cdot,\mu)\rightarrow y_{j,k}(\cdot,\mu_0) \quad\mbox{in} \quad W^{n+r}_p \quad\mbox{as}\quad \mu\to\mu_0.
\end{equation}

For any $k\in\{1, \ldots, rm\}$ and $\mu\in (\mu_0-\varepsilon,\mu_0+\varepsilon)$, we define a vector-valued function $x_k(\cdot,\mu)\in(W^{n+r}_p)^{rm}$ by formula \eqref{yx} in which we replace $x(\cdot;\mu)$ with $x_k(\cdot,\mu)$  and take
$$
y(\cdot;\mu):=\mbox{col}\big(y_{1,k}(\cdot;\mu),\ldots, y_{m,k}(\cdot;\mu)\big).
$$

Let $X(\cdot;\mu)$ denote the matrix-valued function from $(W^{n+r}_p)^{rm\times rm}$ such that its $k$-th column is $x_k(\cdot,\mu)$ for each $k\in\{1, \ldots, rm\}$. This function satisfies the matrix differential equation
\begin{equation}\label{X ym}
X'(t,\mu)+K(t,\mu)X(t,\mu)=O_{rm}, \quad t\in (a,b).
\end{equation}
Therefore, $\det X(t;\mu)\neq0$ for all $t\in[a,b]$, since otherwise the columns of the matrix-valued function $X(\cdot;\mu)$ and, hence, of $Y(\cdot;\mu)$ would be linearly dependent on $[a,b]$, contrary to~\eqref{eq2vk}.
Due to \eqref{zb ym}, we have the convergence $X(\cdot;\mu)\rightarrow X(\cdot;\mu_0)$ in the Banach algebra $(W^{n+r}_p)^{rm\times rm}$ as $\mu\to\mu_0$. Hence, $$\big(X(\cdot;\mu)\big)^{-1}\rightarrow \big(X(\cdot;\mu_0)\big)^{-1}$$ in this algebra. Therefore, in view of~\eqref{X ym}, we conclude that
  $$
   K(\cdot;\mu)=-X'(\cdot;\mu)\big(X(\cdot;\mu)\big)^{-1} \rightarrow -X'(\cdot;\mu_0)\big(X(\cdot;\mu_0)\big)^{-1}=K(\cdot;\mu_0)
  $$
 in $(W^{n}_p)^{rm\times rm}$ as $\mu\to\mu_0$. Thus, the problem \eqref{bound_z1}, \eqref{bound_z2} satisfies Limit Condition~(I). Specifically,
\begin{equation}\label{ob Av}
\big\|A_{r-j}(\cdot,\mu)\big\|_{n,p}=O(1) \quad\mbox{as}\quad \mu\to\mu_0 \quad\mbox{for each}\quad j\in\{1, \ldots, r\}.
\end{equation}

\emph{Step 2.} Let us show that the Limit Condition (II) is satisfied. First, we prove that
\begin{equation}\label{ob Bv}
\|B(\mu)\|=O(1) \quad\mbox{as}\quad \mu\to\mu_0.
\end{equation}
Here, $\|\cdot\|$ denotes the norm of the operator \eqref{oper_B(e)v}.

Suppose the contrary. Then there exists a sequence $\left(\mu^{(k)}\right)_{k=1}^{\infty}\subset(\mu_0-\varepsilon,\mu_0+\varepsilon)$ such that
\begin{equation}\label{ob Bev}
\mu^{(k)}\to\mu_0\quad\mbox{and}\quad 0<\left\|B\bigl(\mu^{(k)}\bigr)\right\|\to\infty \quad\mbox{as}\quad k\rightarrow \infty,
\end{equation}
with $\left\|B\bigl(\mu^{(k)}\bigr)\right\|\neq0$ for all $k\in \mathbb{N}$. For every integer $k\in \mathbb{N}$, we choose a vector-valued function
$\omega_{k}\in(W^{n+r}_p)^{m}$ that satisfies the conditions
\begin{equation}\label{ob wv}
\|\omega_{k}\|_{n+r,p}=1 \quad \mbox{and} \quad \bigl\|B\bigl(\mu^{(k)}\bigr)\omega_{k}\bigr\|_{\mathbb{C}^{rm}}\geq \frac{1}{2}
\bigl\|B\bigl(\mu^{(k)}\bigr)\bigr\|.
\end{equation}

Besides, we put
\begin{gather*}
y\bigl(\cdot;\mu^{(k)}\bigr):=
\bigl\|B\bigl(\mu^{(k)}\bigr)\bigr\|^{-1}\omega_{k} \in (W^{n+r}_p)^{m},\\
f\bigl(\cdot;\mu^{(k)}\bigr):=
L\bigl(\mu^{(k)}\bigr)\,y\bigl(\cdot;\mu^{(k)}\bigr) \in (W^{n}_p)^{m},\\
c\bigl(\mu^{(k)}\bigr):=B\bigl(\mu^{(k)}\bigr)\,y\bigl(\cdot;\mu^{(k)}\bigr) \in \mathbb{C}^{rm}. \end{gather*}

Due to \eqref{ob Bev} and \eqref{ob wv}, we have the convergence
\begin{equation}\label{zb ymal}
y\left(\cdot;\mu^{(k)}\right)\to0 \quad \mbox{in} \quad (W^{n+r}_p)^{m} \quad \mbox{as} \quad k\rightarrow \infty.
\end{equation}

Hence,
\begin{equation}\label{zb fmal}
f\left(\cdot;\mu^{(k)}\right)\to0 \quad \mbox{in}\quad (W^{n}_p)^{m} \quad \mbox{as} \quad k\rightarrow \infty,
\end{equation}
because the problem \eqref{bound_z1}, \eqref{bound_z2} satisfies Limit Condition~(I) (this was proved on Step 1).

Since the finite-dimensional space $\mathbb{C}^{rm}$ is locally compact, then, according to~\eqref{ob wv}, we conclude that
$$1/2\leq\left\|c\left(\mu^{(k)}\right)\right\|_{\mathbb{C}^{rm}}\leq1.$$

Indeed,
\begin{gather*}
\left\|c\left(\mu^{(k)}\right)\right\|_{\mathbb{C}^{rm}}\leq \left\|B\left(\mu^{(k)}\right)\right\| \left\|y\left(\cdot,\mu^{(k)}\right)\right\|_{n+r,p}=\\
=\left\|B\left(\mu^{(k)}\right)\right\| \left\|B\left(\mu^{(k)}\right)\right\|^{-1} \left\|\omega_k\right\|_{n+r,p}=1
\end{gather*}
and
\begin{gather*}
\left\|c\left(\mu^{(k)}\right)\right\|_{\mathbb{C}^{rm}}= \left\|B\left(\mu^{(k)}\right)\left(
\left\|B\left(\mu^{(k)}\right)\right\|^{-1}\omega_k\right)\right\| = \left\|B\left(\mu^{(k)}\right)\right\|^{-1} \left\|B\left(\mu^{(k)}\right)\omega_k\right\|\geq \frac{1}{2}.
\end{gather*}
Hence, there exists a subsequence $$\left(c\left(\mu^{(k_p)}\right)\right)^\infty_{p=1}\subset \left(c\left(\mu^{(k)}\right)\right)^\infty_{k=1}$$ and a nonzero vector $c(\mu_0)\in \mathbb{C}^{rm}$ such that
\begin{equation}\label{zb cmal}
c\left(\mu^{(k_p)}\right)\to c(\mu_0) \quad \mbox{in} \quad \mathbb{C}^{rm} \quad \mbox{as} \quad p\rightarrow \infty.
\end{equation}
For every integer $p\in \mathbb{N}$, the vector-valued function $y\left(\cdot;\mu^{(k_p)}\right)\in\left(W^{n+r}_p\right)^{m}$ is a unique solution to the boundary-value problem
\begin{gather*}
L\left(\mu^{(k_p)}\right)y\left(t;\mu^{(k_p)}\right)=f\left(t;\mu^{(k_p)}\right),  \quad t\in (a,b),\\
B\left(\mu^{(k_p)}\right)y\left(\cdot;\mu^{(k_p)}\right)=c\left(\mu^{(k_p)}\right).
\end{gather*}
Due to \eqref{zb fmal} and \eqref{zb cmal} and condition $(\ast\ast)$ of Definition \ref{defin_vp}, we conclude that the function $y\left(\cdot;\mu^{(k_p)}\right)$ converges to the unique solution $y(\cdot;\mu_0)$ of the boundary-value problem
\begin{gather}
L(\mu_0)y(t,\mu_0)=0,\quad t\in (a,b), \nonumber\\
B(\mu_0)y(\cdot;\mu_0)=c(\mu_0) \label{kr ym}
\end{gather}
in the space $(W^{n+r}_p)^{m}$ as $k\rightarrow \infty$. But $y(\cdot;\mu_0)\equiv0$ due to \eqref{zb ymal}. This contradicts the boundary condition~\eqref{kr ym}, in which $c(\mu_0)\neq0$. Thus, our assumption is false, which proves the required property \eqref{ob Bv}.

\emph{Step 3.} Using the results of the previous steps, we will prove here that the problem \eqref{bound_z1}, \eqref{bound_z2} satisfies Limit Condition (II). According to \eqref{ob Av} and \eqref{ob Bv}, there exist numbers $\gamma'>0$ and $\varepsilon'>0$ such that
\begin{equation}\label{ner 1}
\|(L(\mu),B(\mu))\|\leq\gamma' \quad \mbox{for every} \quad \mu\in(\mu_0-\varepsilon',\mu_0+\varepsilon').
\end{equation}
Here, $\|\cdot\|$ denotes the norm of the bounded operator \eqref{(L,B)vp}. We arbitrarily choose a vector-valued function $y\in(W^{n+r}_p)^{m}$ and set $f(\cdot;\mu):=L(\mu)y$ and $c(\mu):=B(\mu)y$ for every $\mu\in(\mu_0-\varepsilon',\mu_0+\varepsilon')$. Hence,
\begin{equation}\label{riv 1}
y=(L(\mu),B(\mu))^{-1}(f(\cdot;\mu),c(\cdot;\mu)) \quad \mbox{for every} \quad \mu\in(\mu_0-\varepsilon',\mu_0+\varepsilon').
\end{equation}
Here, $(L(\mu),B(\mu))^{-1}$ denotes the operator inverse to the operator \eqref{(L,B)vp}. The latter operator is invertible due to condition $(\ast)$ of Definition \ref{defin_vp}.

Using \eqref{ner 1} and \eqref{riv 1}, we obtain the following inequalities for every $\mu\in(\mu_0-\varepsilon',\mu_0+\varepsilon')$:
\begin{gather*}
\bigl\|B(\mu)y-B(\mu_0)y\bigr\|_{\mathbb{C}^{rm}}
\leq\bigl\|(f(\cdot;\mu),c(\mu))-
(f(\cdot;\mu_0),c(\mu_0))\bigr\|_{(W^{n}_p)^{m}\times\mathbb{C}^{rm}}=\\
=\bigl\|(L(\mu),B(\mu))(L(\mu),B(\mu))^{-1}(f(\cdot;\mu),c(\mu))-
(f(\cdot;\mu_0),c(\mu_0))\bigr\|_{(W^{n}_p)^{m}\times\mathbb{C}^{rm}}\leq\\
\leq\gamma'\,\bigl\|(L(\mu),B(\mu))^{-1}
\bigl((f(\cdot;\mu),c(\mu))-
(f(\cdot;\mu_0),c(\mu_0))\bigr)\bigr\|_{n+r,p}=\\
=\gamma'\,\bigl\|(L(\mu_0),B(\mu_0))^{-1}(f(\cdot;\mu_0),c(\mu_0))
-(L(\mu),B(\mu))^{-1}(f(\cdot;\mu_0),c(\mu_0))\bigr\|_{n+r,p}.
\end{gather*}

The latter norm vanishes as $\mu\to\mu_0$, according to condition $(\ast\ast)$ of Definition \ref{defin_vp}. Since $$\|B(\mu)\|=O(1) \quad \mbox{and} \quad \bigl\|B(\mu)y-B(\mu_0)y\bigr\|_{\mathbb{C}^{rm}} \to0,$$ we have the convergence
 $B(\mu)y$ to $B(\mu_0)y$ in $\mathbb{C}^{rm}$ as $\mu\to\mu_0$ for all $y\in(W^{n+r}_p)^m$.
We conclude that the boundary-value problem \eqref{bound_z1}, \eqref{bound_z2} satisfies Limit Condition~(II).

The necessity is proved.

Let us prove \emph{the sufficiency}. Suppose that the boundary-value problem \eqref{bound_z1}, \eqref{bound_z2} satisfies Condition (0) and Limit Conditions (I) and (II). We show that the solution of this problem  continuously depends on the~parameter $\mu$ at $\mu_0$ in the space $(W^{n+r}_p)^{m}$. We divide our proof into four steps.

\emph{Step 1.} For every $\mu\in I$, we first consider Cauchy problem
\begin{gather}\label{s1}
L(\mu)y(t,\mu)=f(t,\mu),\quad t\in(a,b),\\
\hat{y}^{(j-1)}(a,\mu)=c_{j}(\mu),\quad j\in\{1,\ldots,r\}. \label{ku1}
\end{gather}
Here, for every $\mu$, the vector-valued function $f(\cdot,\mu)\in(W^{n}_p)^m$ and the vectors $c_{j}(\mu)\in\mathbb{C}^{rm}$ are arbitrarily chosen. The unique solution $\hat{y}(\cdot,\mu)$ of this problem belongs to the space $(W^{n+r}_p)^{m}$.

We show that the convergence of the right-hand sides of this problem
\begin{equation}\label{r_s.1}
f(\cdot,\mu)\to f(\cdot,\mu_0)\quad\mbox{in}\quad(W^{n}_p)^{m} \quad\mbox{as}\quad \mu\to \mu_0,
\end{equation}
\begin{equation}\label{gsp}
\begin{gathered}
c_{j}(\mu)\rightarrow c_{j}(\mu_0)\quad \mbox{in} \quad \mathbb{C}^{rm}\quad\mbox{as}\quad \mu\to \mu_0
\quad\mbox{for every}\quad j\in\{1,\ldots,r\}
\end{gathered}
\end{equation}
implies the convergence of its solutions
\begin{equation}\label{gsx}
\hat{y}(\cdot,\mu)\to \hat{y}(\cdot,\mu_0)\quad\mbox{in}\quad(W^{n+r}_p)^{m} \quad\mbox{as}\quad \mu\rightarrow\mu_0.
\end{equation}

We reduce the Cauchy problem \eqref{s1}, \eqref{ku1} to the following Cauchy problem for the system of differential equations of the first order:
\begin{gather}\label{Cauchi-reduced-DE}
x'(t,\mu)+K(t,\mu)x(t,\mu)=g(t,\mu),
\quad t\in(a,b),\\
x(a,\mu)=c(\mu).\label{Cauchi-reduced-cond}
\end{gather}
Here, the matrix-valued function $K(\cdot,\mu)$ and the vector-valued function $g(\cdot,\mu)$ are the same as in Step 1 of proof of the necessity. Moreover,
\begin{gather*}\label{ghggh}
x(\cdot,\mu):=\mbox{col}\left(\hat{y}(\cdot,\mu),\hat{y}'(\cdot,\mu),\ldots,\hat{y}^{(r-1)}(\cdot,\mu)\right)\in (W^{n+r}_p)^{rm}, \\
c(\mu):=\mbox{col}\left(c(\mu),\ldots,c_r(\mu)\right)\in \mathbb{C}^{rm}.
\end{gather*}

Since, by assumption, the boundary-value problem~\eqref{bound_z1}, \eqref{bound_z2} satisfies Limit Condition (I), we get
\begin{equation*}\label{Kgsx}
K(\cdot,\mu)\to K(\cdot,\mu_0)\quad\mbox{in}\quad(W^{n}_p)^{rm\times rm} \quad\mbox{as}\quad \mu\rightarrow\mu_0.
\end{equation*}
The conditions \eqref{r_s.1} and \eqref{gsp} imply the convergence of the right-hand sides of the problem \eqref{Cauchi-reduced-DE}, \eqref{Cauchi-reduced-cond}: \begin{equation*}\label{r_g.1}
g(\cdot,\mu)\to g(\cdot,\mu_0)\quad\mbox{in}\quad(W^{n}_p)^{rm} \quad\mbox{as}\quad \mu\to\mu_0,
\end{equation*}
\begin{equation*}\label{ggp}
\begin{gathered}
c(\mu)\rightarrow c(\mu_0)\quad\mbox{in}\quad\mathbb{C}^{rm}\quad\mbox{as}\quad \mu\to \mu_0.
\end{gathered}
\end{equation*}
Therefore, due to \cite[Theorem~1]{AtlasiukMikhailets20192}, \eqref{r_s.1} and \eqref{gsp}, we have the convergence \eqref{gsx}.

\emph{Step 2.} We prove that the boundary-value problem~\eqref{bound_z1}, \eqref{bound_z2} satisfies the condition $(\ast)$ of  Definition \ref{defin_vp}, that is, for sufficiently small $|\mu-\mu_0|$, the operator $(L(\mu),B(\mu))$ is invertible.

For every $\mu\in I$ and $k\in\{0,\ldots,r-1\}$, we consider matrix Cauchy problem
\begin{gather*}\label{zad kosh1v}
Y_k^{(r)}(t,\mu)+\sum\limits_{j=1}^rA_{r-j}(t,\mu)Y_k^{(r-j)}(t,\mu)=O_{m},\quad t\in (a,b),
\end{gather*}
with initial conditions
\begin{gather*}\label{zad kosh2v}
Y_k^{(j)}(t_0,\mu) = \delta_{k,j}I_m,\quad j \in \{1,\dots, r-1\}.
\end{gather*}

 Here, $$Y_k(t,\mu)=\left(y_k^{\alpha,\beta}(t,\mu)\right)_{\alpha,\beta=1}^m$$ is an unknown $m\times m$ matrix-valued function, the point $t_0 \in [a,b]$ is fixed, $\delta_{k,j}$ is the Kronecker symbol, $O_{m}$ and $I_m$ are, respectively, the zero and identity matrices of the order $m$. It consist of $m$ Cauchy problems of the form \eqref{s1}, \eqref{ku1} with $f=0$ for the vector-valued function $\hat{y}(\cdot,\mu)$, that are columns of the matrix $Y_k(\cdot,\mu)$.

We write the solution of the homogeneous differential equation \eqref{bound_z1} in the form
\begin{equation}\label{solution_by_Yv}
y(\cdot,\mu)=\sum\limits_{k=0}^{r-1}Y_k(\cdot,\mu)q_k(\mu),
\end{equation}
where $q_k(\mu)\in\mathbb{C}^{rm}$ are arbitrary column vectors \cite{Cartan1971}.

Since the right-hand sides of this problem do not depend on $\mu$, then
\begin{equation}\label{Zlimitv}
Y_l(\cdot,\mu)\to Y_l(\cdot,\mu_0) \quad\mbox{in}\quad(W^{n+r}_p)^{m\times m}\quad\mbox{as}\quad \mu\to\mu_0,
\end{equation}
according to Step 1. In view of Limit Condition~(II), this yields the convergence of the following block numerical matrices:
\begin{equation}\label{BZlimit}
\begin{gathered}
\bigl([B(\mu)Y_0(\cdot,\mu)],\ldots,
[B(\mu)Y_{r-1}(\cdot,\mu)]\bigr)\to\\
\to\bigl([B(\mu_0)Y_0(\cdot,\mu)],\ldots,
[B(\mu_0)Y_{r-1}(\cdot,\mu)]\bigr)\quad \mbox{in} \quad \mathbb{C}^{rm\times rm}\quad
\mbox{as}\quad\mu\to\mu_0.
\end{gathered}
\end{equation}
However, the limit matrix is nondegenerate by virtue of Limit Condition (0) and \cite[Lemma 1]{AtlasiukMikhailets20192}. Therefore, there is a positive number $\varepsilon_1$ such that, for every  $\mu\in(\mu_0-\varepsilon_1,\mu_0+\varepsilon_1)$, we have
\begin{equation}\label{BZdet}
\begin{gathered}
\det\bigl(M(L(\mu),B(\mu))\bigr)\neq0.
\end{gathered}
\end{equation}
Hence, by \cite[Lemma 1]{AtlasiukMikhailets20192}, the operator \eqref{(L,B)vp} is invertible.

\emph{Step 3.} We show that the boundary-value problem~\eqref{bound_z1}, \eqref{bound_z2} satisfies the condition $(\ast\ast)$ of Definition \ref{defin_vp}. Let us first analyze the case $f(\cdot,\mu)\equiv0$.

Consider a semihomogeneous boundary-value problem
 \begin{equation}\label{vv}
    L(\mu)v(\cdot;\mu)\equiv 0,
\end{equation}
 \begin{equation}\label{vvk}
 B(\mu)v(\cdot;\mu) =c(\mu),
\end{equation}
depending on the parameter $\mu$. According to Step 2, this problem has a unique solution  $v(\cdot;\mu)\in(W^{n+r}_p)^{m}$ for every $\mu\in(\mu_0-\varepsilon',\mu_0+\varepsilon')$, where $\varepsilon$ is the sufficiently small positive number. Suppose that
 \begin{equation}\label{zb c2v}
  c(\mu)\rightarrow c(\mu_0) \quad \mbox{in} \quad \mathbb{C}^{rm} \quad \mbox{as} \quad\mu\rightarrow \mu_0.
 \end{equation}
Let us show that
 \begin{equation}\label{zb v2v}
  v(\cdot,\mu)\rightarrow v(\cdot,\mu_0) \quad \mbox{in} \quad (W^{n+r}_p)^{m} \quad \mbox{as} \quad\mu\rightarrow \mu_0.
 \end{equation}

 For every $\mu\in(\mu_0-\varepsilon',\mu_0+\varepsilon')$, we write the general solution of the homogeneous differential equation \eqref{vv} in the form \eqref{solution_by_Yv}, that is
\begin{equation}\label{solution_by_vv}
v(\cdot,\mu)=\sum\limits_{k=0}^{r-1}Y_k(\cdot,\mu)q_k(\mu),
\end{equation}
where $q_k(\mu)\in\mathbb{C}^{rm}$ are arbitrary column vectors, and each matrix-valued function $Y_k(\cdot,\mu)\in(W^{n+r}_p)^{m\times m}$ as in Step 2. By virtue of \cite[Lemma 6]{AtlasiukMikhailets20191}, we have
\begin{equation*}\label{rivn_matruz vv}
B(\mu)v(\cdot,\mu)=\sum _{k=0}^{r-1}B(\mu)(Y_k(\cdot,\mu)q_k(\mu)) = \sum _{k=0}^{r-1}\left[B(\mu)Y_k(\cdot,\mu)\right]q_k(\mu).
\end{equation*}
Therefore, the boundary-value problem \eqref{vvk} is equivalent to
\begin{equation*}\label{rivn_matruz vv}
\sum _{k=0}^{r-1}\left[B(\mu)Y_k(\cdot,\mu)\right]q_k(\mu)=c(\mu).
\end{equation*}
The last condition can be rewritten in the form of a system of linear algebraic equations
\begin{equation*}\label{rivn_matruz vv}
\big(\left[B(\mu)Y_0(\cdot,\mu)\right], \ldots, \left[B(\mu)Y_{r-1}(\cdot,\mu)\right]\big)q(\mu)=c(\mu),
\end{equation*}
that is
\begin{equation*}\label{rivn_matruz vv}
\left[B(\mu)Y(\cdot,\mu)\right]q(\mu)=c(\mu)
\end{equation*}
for the column vector $$q(\mu):=\mbox{col}(q_0(\mu),\ldots, q_{r-1}(\mu)).$$ By virtue of relations \eqref{BZlimit}, \eqref{BZdet} and assumption \eqref{zb c2v}, we have the convergence
\begin{equation*}\label{rivn_matruz vv}
q(\mu)\left[B(\mu)Y(\cdot,\mu)\right]^{-1}c(\mu)\rightarrow\left[B(\mu_0)Y(\cdot,\mu_0)\right]^{-1}c(\mu_0)=q(\mu_0) \quad \mbox{as} \quad\mu\to \mu_0.
\end{equation*}
Therefore, the necessary convergence \eqref{zb v2v} follows from the formulas \eqref{Zlimitv}, \eqref{solution_by_vv}, namely
\begin{equation*}\label{rivn_matruz vv}
v(\cdot,\mu)=\sum _{k=0}^{r-1}Y_k(\cdot,\mu)q_k(\mu) \rightarrow \sum _{k=0}^{r-1}Y_k(\cdot,\mu_0)q_k(\mu_0)=v(\cdot,\mu_0) \quad \mbox{in} \quad (W^{n+r}_p)^{m} \quad \mbox{as} \quad\mu\to \mu_0.
\end{equation*}

\emph{Step 4.} We now turn to the general case of the inhomogeneous differential equation \eqref{bound_z1}. Suppose that the conditions \eqref{zb c2v} and
\begin{equation}\label{zb fvvy}
f(\cdot,\mu) \rightarrow f(\cdot,\mu_0) \quad \mbox{in} \quad (W^{n}_p)^{m} \quad \mbox{as} \quad\mu\to \mu_0
\end{equation}
are satisfied.
For every $\mu \in (\mu_0-\varepsilon_1,\mu_0+\varepsilon_1)$, we set
$$z(\cdot;\mu)=y(\cdot;\mu)-\hat{y}(\cdot;\mu),$$
where vector-valued function $y(\cdot;\mu)$ is a solution of the inhomogeneous boundary-value problem \eqref{bound_z1}, \eqref{bound_z2} and the vector-valued function $\hat{y}(\cdot;\mu)$ is a solution of the Cauchy problem \eqref{s1}, \eqref{ku1}, with $c_{j}(\mu)\equiv0$. Then $z(\cdot;\mu)$ is a solution of the semihomogeneous boundary-value problem
\begin{gather*}\label{s2}
L(\mu)z(\cdot,\mu)\equiv0,\\
B(\mu)z(\cdot,\mu)=\tilde{c}(\mu),\\
\tilde{c}(\mu)=c(\mu)-B(\mu)\hat{y}(\cdot,\mu) \in \mathbb{C}^{rm}.
\end{gather*}
In Step 1, it was shown that $\hat{y}(\cdot,\mu)$ satisfies the property \eqref{gsx} if the condition \eqref{zb fvvy} is fulfilled. By virtue of this property and the assumption that the boundary-value problem~\eqref{bound_z1}, \eqref{bound_z2} satisfies Limit Condition (I) and \eqref{zb c2v}, we get
 \begin{equation*}\label{zb c2tv}
  \tilde{c}(\mu)\rightarrow \tilde{c}(\mu_0) \quad \mbox{in} \quad \mathbb{C}^{rm} \quad \mbox{as} \quad\mu\rightarrow \mu_0.
 \end{equation*}
Hence, according to Step 3, we conclude that
\begin{equation*}\label{zb fvy}
z(\cdot,\mu) \rightarrow z(\cdot,\mu_0) \quad \mbox{in} \quad (W^{n+r}_p)^{m} \quad \mbox{as} \quad\mu\to \mu_0.
\end{equation*}
Therefore, from the formula \eqref{gsx}, we obtain the necessary convergence
\begin{gather*}\label{zb fvy}
y(\cdot,\mu)= \hat{y}(\cdot,\mu)+z(\cdot,\mu)\rightarrow \hat{y}(\cdot,\mu_0)+z(\cdot,\mu_0)=y(\cdot,\mu_0)\\ \quad \mbox{in} \quad (W^{n+r}_p)^{m} \quad \mbox{as} \quad\mu\to \mu_0.
\end{gather*}

The sufficiency is proved.

\end{proof}

\begin{proof}[Proof of Theorem~\ref{3.6.th-bound v}]
Let us first prove the left-hand side of~\eqref{3.6.bound}. Limit Conditions (I) and (II) imply the strong convergence
$$
(L(\mu),B(\mu))\stackrel{s}{\longrightarrow}(L(\mu_0),B(\mu_0))
\quad\mbox{as}\quad \mu\to\mu_0
$$
of the continuous operators acting from $(W^{n+r}_p)^{m}$ to $(W^{n}_p)^{m}\times \mathbb{C}^{rm}$. Hence, there exist numbers $\gamma'>0$ and $\varepsilon >0$ such that the norm of this operator satisfies inequality \eqref{ner 1} if $|\mu -\mu_0|<\varepsilon$.
Indeed, if we assume the contrary, then we can find a sequence of positive numbers $\left(\mu^{(k)}\right)_{k=1}^{\infty}$ such that
\begin{gather*}
\mu^{(k)}\to0\quad\mbox{and}\quad\left\|\left(L\left(\mu^{(k)}\right),
B\left(\mu^{(k)}\right)\right)\right\|\to\infty
\quad\mbox{as}\quad k\to\infty.
\end{gather*}
However, by the Banach -- Steinhaus theorem, this contradicts the fact that
$$\left(L\left(\mu^{(k)}\right)B\left(\mu^{(k)}\right)\right) \stackrel{s}{\longrightarrow} \left(L(\mu_0), B(\mu_0)\right) \quad \text{as} \quad k\to\infty.$$

Now, based on inequality \eqref{ner 1}, we conclude that
\begin{gather*}
\widetilde{d}_{n,p}(\mu)=\bigl\|L(\mu)y(\cdot;\mu_0)-L(\mu)y(\cdot;\mu)\bigr\|_{n,p}+
\bigl\|B(\mu)y(\cdot;\mu_0)-B(\mu)y(\cdot;\mu)\bigr\|_{\mathbb{C}^{rm}}\leq\\
\leq\bigl\|L(\mu)\bigr\|\bigl\|y(\cdot;\mu_0)-y(\cdot;\mu)\bigr\|_{n+r,p}+
\bigl\|B(\mu)\bigr\|\bigl\|y(\cdot;\mu_0)-y(\cdot;\mu)\bigr\|_{n+r,p}
\leq \\ \leq \gamma'\bigl\|y(\cdot;\mu_0)-y(\cdot;\mu)\bigr\|_{n+r,p},
\end{gather*}
for every small $|\mu-\mu_0|$. Thus, we have established the left-hand side of the estimate \eqref{3.6.bound}, where $\gamma_{1}:=1/\gamma'$.

Let us prove the right-hand side of the estimate \eqref{3.6.bound}. By Theorem~\ref{nep v}, the boundary-value problem \eqref{bound_z1}, \eqref{bound_z2} satisfies Definition~\ref{defin_vp}. Therefore, the operator \eqref{(L,B)vp} is invertible for every small $|\mu-\mu_0|$. Moreover, we have the following strong convergence
$$
(L(\mu),B(\mu))^{-1} \stackrel{s}{\longrightarrow} (L(\mu_0),B(\mu_0))^{-1}, \quad \mu\to\mu_0.
$$
Indeed, for arbitrary $f\in(W^{n}_p)^{m}$ and $c\in\mathbb{C}^{rm}$, under the condition $(\ast\ast)$ of Definition~\ref{defin_vp}, we get the following convergence
\begin{equation*}
(L(\mu),B(\mu))^{-1}(f,c)=:y(\cdot,\mu)\to
y(\cdot,\mu_0):=(L(\mu_0),B(\mu_0))^{-1}(f,c)
\end{equation*}
in space $(W^{n+r}_p)^{m}$ as $\mu\to\mu_0$.

Hence, by the Banach -- Steinhaus theorem, the norms of these inverse operators are bounded, namely, there exist positive numbers $\varepsilon$ and $\gamma_{2}$ such that the norm of the inverse operator
$$
\left\|(L(\mu),B(\mu))^{-1}\right\|\leq\gamma_{2} \quad\mbox{for every}\quad \gamma\in(\mu_0-\varepsilon,\mu_0+\varepsilon).
$$
Thus, for every $\mu\in(\mu_0-\varepsilon,\mu_0+\varepsilon)$, we conclude that
\begin{gather*}
\big\|y(\cdot,\mu_0)-y(\cdot,\mu)\big\|_{n+r,p}
=\|(L(\mu),B(\mu))^{-1}(L(\mu),B(\mu))
(y(\cdot,\mu_0)-y(\cdot,\mu))\|_{n+r,p}\leq\\
\leq\gamma_{2}\,\|(L(\mu),B(\mu))
(y(\cdot,\mu_0)-y(\cdot,\mu))\|_{(W^{n}_p)^{m}\times\mathbb{C}^{rm}}=\gamma_{2}\,\widetilde{d}_{n,p}(\mu).
\end{gather*}
This directly yields the right-hand side of the two-sided estimate \eqref{3.6.bound}.
\end{proof}

\section{Acknowledgments}

The work of the first named author was funded by the Isaac Newton Institute of Mathematical Sciences "Solidarity Program" \, and the London Mathematical Society. The author wishes to thank the Department of Mathematics, King's College London, for their hospitality.

The work of the second named author was funded by Postdoctoral Fellowship EU-MSCA4Ukraine (number: 1244691, WBS-number: 4100609). This project has received funding through the MSCA4Ukraine project, which is funded by the European Union. Views and opinions expressed are however those of the author only and do not necessarily reflect those of the European Union, the European Research Executive Agency or the MSCA4Ukraine Consortium. Neither the European Union nor the European Research Executive Agency, nor the MSCA4Ukraine Consortium as a whole nor any individual member institutions of the MSCA4Ukraine Consortium can be held responsible for them.”

\end{document}